\documentclass[a4paper]{article}

\usepackage{amsmath,amssymb,amsthm}
\usepackage{mathrsfs}

\usepackage[left=2cm,right=2cm,top=2cm,bottom=2cm,bindingoffset=0cm]{geometry}
\usepackage{nccfoots}
\usepackage[hang,flushmargin]{footmisc} 
\usepackage{hyperref}

\usepackage{pgf,tikz}

\newtheorem{lemma}{Lemma}
\newtheorem{cor}{Corollary}

\newtheorem{theorem}{Theorem}

\newcommand{\HH}{{\mathcal H}}
\renewcommand{\H}{\HH^1}

\def\dist{\mathrm{dist}\,}

\title{A self-similar infinite binary tree is a solution of Steiner problem}

\author{D. Cherkashin\footnote{Institute of Mathematics and Informatics, Bulgarian Academy of Sciences. The work is supported by the Ministry of Education and Science of Bulgaria, Scientific Programme ``Enhancing the Research Capacity in Mathematical Sciences (PIKOM)'', No. DO1-67/05.05.2022}~~and Y. Teplitskaya
\footnote{Mathematical Institute, Leiden University, Netherlands}}

\begin{document}

\maketitle

\begin{abstract}

We consider a general metric Steiner problem which is of finding a set $\mathcal{S}$ with minimal length such that $\mathcal{S} \cup A$ is connected, where $A$ is a given compact subset of a given complete metric space $X$; a solution is called Steiner tree.

Paper~\cite{paolini2015example} contains an example of a planar Steiner tree with an infinite number of branching points connecting
an uncountable set of points. We prove that such a set can have a positive Hausdorff dimension which was an open question (the corresponding tree is a self-similar fractal).

\end{abstract}



\section{Introduction}

The Steiner problem which has various different but more or less equivalent formulations, is that of finding a set $\mathcal{S}$ 
with minimal length (one-dimensional Hausdorff measure $\H$) such that $\mathcal{S} \cup A$ is connected, where $A$ is a given compact subset of a given complete metric space $X$. 

It is shown in~\cite{paolini2013existence}, that under rather mild assumptions on the ambient space $X$ (which anyhow are true in the Euclidean plane setting) it is shown that a solution of the Steiner problem exists. 
Moreover every solution $\mathcal{S}$ having a finite length has the following properties:
\begin{itemize}
    \item $\mathcal{S} \cup A$ is compact,
    \item $\mathcal{S} \setminus A$ has at most countably many connected components, and each of the latter has strictly positive length,
    \item $\overline{\mathcal{S}}$ contains no loops (homeomorphic images of the circle $\mathbb{S}^1$),
    \item the closure of every connected component of $\mathcal{S}$ is a topological tree (a connected, locally connected
    compact set without loops) with endpoints on $A$ (so that in particular it has at most countable number
    of branching points), and with at most one endpoint on each connected component of $A$ and all the
    branching points having finite order (i.e. finite number of branches leaving them),
    \item if $A$ has a finite number of connected components, then $\mathcal{S} \setminus A$ also has finitely many connected components,
    the closure of each of which is a finite geodesic embedded graph with endpoints on $A$, and with at most
    one endpoint on each connected component of $A$,
    \item for every open set $U \subset X$ such that $A \subset U$ one has that the set $\mathcal{S}' := \mathcal{S} \setminus U$ is a subset of a finite geodesic embedded graph. Moreover, for a.e. $\varepsilon > 0$ one has that for $U = \{x : \dist(x, A) < \varepsilon\}$ the set $\mathcal{S}'$ is a finite geodesic embedded graph (in particular, it has a finite number of connected components and a finite number of branching points).
\end{itemize}

A solution of the Steiner problem is called \textit{Steiner tree}; the above properties explain such naming in the case of $A$ being a totally disconnected set.
Denote by $\mathbb{M}(A)$ the set of Steiner trees for $A$.
From now we focus on $X = \mathbb{R}^2$ (but the following properties are also hold for $\mathbb{R}^d$) and assume that $A$ is a totally disconnected set; we call the points from $A$ \textit{terminals}. In this case a geodesic is just a straight line segment. 
A combination of the last enlisted property from~\cite{paolini2013existence} with a well-known facts on Euclidean Steiner trees (see~\cite{gilbert1968steiner,hwang1992steiner}) gives the following properties.
The maximal degree (in graph-theoretic sense) of a vertex in a Steiner tree is at most $3$. Moreover, only terminals can have degree $1$ or $2$, all the other vertices have degree $3$ and are called \textit{Steiner points}. Vertices of the degree $3$ are called \textit{branching points}. The angle between any two adjacent edges of a Steiner tree is at least $2\pi/3$.

Note that a Steiner tree may be not unique, see the left-hand side of Fig.~\ref{fig:1nonunique}.
\begin{figure}
    \centering
    \includegraphics[width=\textwidth]{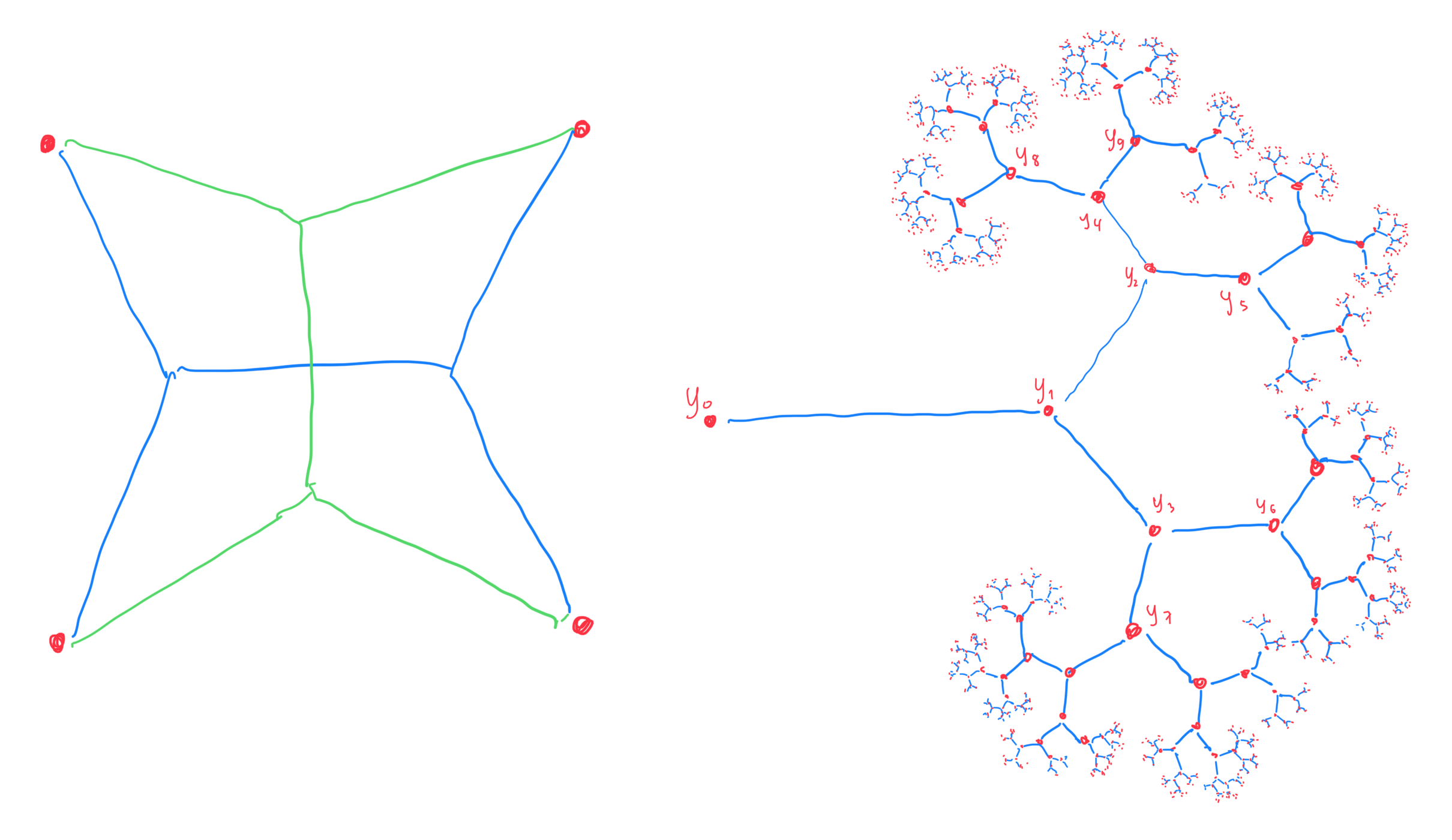}
    \caption{The left part contains two Steiner trees connecting vertices of a square; the right part provides an example of $\Sigma(\Lambda)$}
    \label{fig:1nonunique}
\end{figure}

Branching and in particular triple-branching points is a well-known phenomenon in one-dimensional shape optimization problems.
There are several papers from the last decade concerning the case of an infinite number of branching points for the various forms of the irrigation problem~\cite{morel2010regularity,marchese2016optimal,goldman2017self}. The corresponding result for the Steiner problem was proved in~\cite{paolini2015example} (see Theorem~\ref{th_steinMain0}).

If every terminal point in a Steiner tree has degree one, then it is called \textit{full}. 
In the Steiner tree problem it is reasonable to focus on full trees, since every non-full tree can be easily cut on full components.
Vice versa, it is relatively easy to glue several regular tripods (a \textit{regular tripod} is a union of three segments with a common end and pairwise angles equal to $2\pi/3$) by terminals and get a trivial example of a tree with an infinite number of branching points.

\subsection{A universal Steiner tree}

In this subsection we provide the construction of a unique Steiner tree with an infinite number of Steiner points from~\cite{paolini2015example}. 

Let $S_\infty$ be an infinite tree with vertices $y_0, y_1, y_2,\dots$ and edges given by $y_0y_1$ and $y_ky_{2k},\ y_ky_{2k+1},\ k\geq 1$. Thus, $S_\infty$ is an infinite binary tree with an additional vertex $y_0$ attached to the common parent $y_1$ of all other vertices $y_k,k\geq 2$. The goal of~\cite{paolini2015example} is to embed $S_\infty$ in the plane in such a way that the image of each finite subtree of $S_\infty$ will be the unique Steiner tree for the set of its vertices having degree 1 or 2. We define the embedding below by specifying the positions of $y_0,y_1,y_2,\dots$ on the plane.

Let $\Lambda = \{\lambda_i\}_{i=0}^\infty$ be a sequence of positive real numbers.
Define an embedding $\Sigma (\Lambda)$ of $S_\infty$ as a rooted binary tree with the root $y_0 = (0,0)$ the first descendant $y_1 = (1,0)$ and the ratio between edges of $(i+1)$-th and $i$-th levels being $\lambda_i$.
For a small enough $\{\lambda_i\}$ the set $\Sigma (\Lambda)$ is a tree, see the right-hand side of Fig.~\ref{fig:1nonunique}.

Let $A_\infty (\Lambda)$ be the union of the set of all leaves (limit points) of $\Sigma (\Lambda)$ and $\{y_0\}$.

\begin{theorem}[Paolini--Stepanov--Teplitskaya,~\cite{paolini2015example}]
\label{th_steinMain0}
A binary tree $\Sigma(\Lambda)$ is a unique Steiner tree for $A_\infty (\Lambda)$ provided by $\lambda_i < 1/5000$ and $\sum_{i=1}^\infty \lambda_i  <\pi/5040$.
\end{theorem}

The following corollary explains why a full binary Steiner tree is universal, i.e. it contains a subtree with a given combinatorial structure.

\begin{cor}
In the conditions of Theorem~\ref{th_steinMain0} each connected closed subset $S$ of $\Sigma(\Lambda)$ has a natural tree structure. Moreover, every such an $S$ is the unique Steiner tree for the set $P$ of the vertices with the degree $1$ and $2$ of $S$.
\label{coruniquesubtree}
\end{cor}
\begin{proof}
    Let $S\subset \Sigma(\Lambda)$ and $P\subset S$ satisfy the conditions of the corollary. The fact that $S$ is a tree is straightforward. Let $S'\neq S$ be any Steiner tree for $S$ and assume that $\H(S') \leq \H(S)$. Then it is clear that $\H((\Sigma(\Lambda)\setminus S)\cup S') \leq \H(\Sigma(\Lambda))$, but on the other hand $\{y_0\}\cup A_\infty \subset (\Sigma(\Lambda)\setminus S)\cup S'$, which contradicts to the minimality or the uniqueness in Theorem~\ref{th_steinMain0}.
\end{proof}

After proving Theorem~\ref{th_steinMain0} the authors of~\cite{paolini2015example} asked the following question:

\begin{quote}
Our proof requires that the sequence $\{ \lambda_j \}$ vanish rather quickly (in fact, at least
be summable). It is an open question if in the case of a constant sequence $\lambda _j = \lambda$ (with $\lambda > 0$ small enough)
the same construction still provides a Steiner tree. This seems to be quite interesting since the resulting tree
would be, in that case, a self-similar fractal.
\end{quote}

The result of this paper is an affirmative answer.

\begin{theorem}
A binary tree $\Sigma(\Lambda)$ is a Steiner tree for $A_\infty (\Lambda)$ provided by a constant sequence $\lambda_i = \lambda < \frac{1}{300}$.
\label{newMainT}
\end{theorem}

Clearly, in the conditions of Theorem~\ref{newMainT}
\[
\H(\Sigma(\Lambda)) = \sum_{i=0}^\infty (2\lambda)^i = \frac{1}{1 - 2\lambda}
\]
and the Hausdorff dimension of $A_\infty$ is $-\frac{\ln 2}{\ln \lambda}$.

The set of descendants of every vertex $y_k$ of $\Sigma(\Lambda)$ has an axis of symmetry containing $y_k$.
The idea of the proof of Theorem~\ref{newMainT} is to progressively show that there is a Steiner tree for $A_\infty$ which has more and more such symmetries and then to pass to a limit. 

In fact the proof uses soft analysis in opposite to the previous one, which used hard analysis. The proof uses symmetry arguments instead of stability arguments (the proof of Theorem~\ref{th_steinMain0} is based on a general estimation of a difference after a small perturbation). 
In fact, the proof of Theorem~\ref{th_steinMain0} allows us to use different $\lambda_i$ for different branches, which completely breaks the symmetries used in the proof of Theorem~\ref{newMainT}.

Another weakness by comparison with Theorem~\ref{th_steinMain0} is that we are not able to show the uniqueness of a Steiner tree.
Summing up, to prove or to disprove that a full binary tree with uniformly bounded ratios of consecutive edges is a unique Steiner tree for its terminals is a challenging open question.

\section{Proof of Theorem~\ref{newMainT}}

Let $\lambda < \frac{1}{300}$, $\varepsilon = \frac{\lambda^2}{1 - \lambda}$ be fixed during the proof.
Let $Y_1B_1C_1$ be an isosceles triangle with Fermat--Torricelli point $T_1$ such that
$|Y_1T_1| = 1$, $|T_1B_1| = |T_1C_1| = \lambda$; then by the cosine rule $|Y_1B_1| = |Y_1C_1| = \sqrt{1+\lambda + \lambda^2}$ and $|B_1C_1| = \sqrt{3}\lambda$. Analogously let $Y_2B_2C_2$ be an isosceles triangle with Fermat--Torricelli point $T_2$ such that
$|Y_2T_2| = 1/4$, $|T_2B_2| = |T_2C_2| = \lambda$; then $|Y_2B_2| = |Y_2C_2| = \sqrt{1/16 + \lambda/4 + \lambda^2}$ and $|B_2C_2| = \sqrt{3}\lambda$.
 
\begin{figure}[h]
    \centering
    \includegraphics[width=\textwidth]{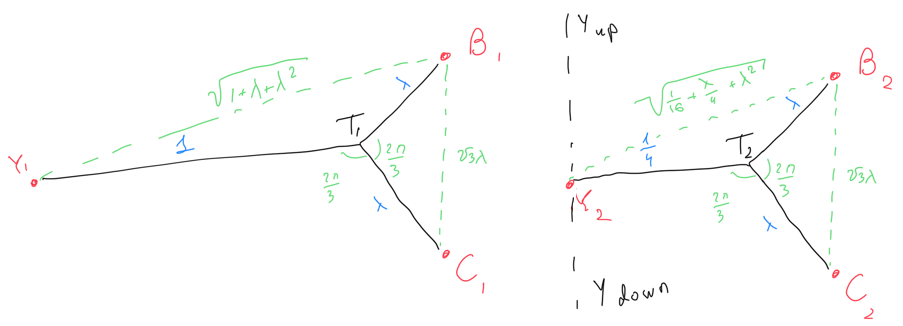}
    \caption{The construction of triangles in lemmas}
    \label{fig:Auxiliary}
\end{figure}
 
Let $b_i \subset B_\varepsilon(B_i)$ and $c_i \subset B_\varepsilon(C_i)$ be symmetric sets with respect to the axis of symmetry $l_i$ of $Y_iB_iC_i$, where $i=1,2$.
Finally let $Y_{up}$, $Y_{down}$ be such points that $Y_{up}Y_{down} \parallel B_2C_2$, $Y_2 \in [Y_{up}Y_{down}]$ and $|Y_2Y_{up}|=|Y_2Y_{down}| = 1/2$.

The following lemma is more-or-less known (see, for instance Lemma A.6 in~\cite{paolini2015example}) so we prove it in Appendix.
Recall that a \textit{regular tripod} is a union of three segments with a common end and pairwise angles equal to $2\pi/3$.

\setcounter{lemma}{-1}

\begin{lemma}
\label{lemma0}
    \begin{itemize}
        \item [(i)] For every $\mathcal{S} \in \mathbb{M}(\{Y_1\} \cup b_1 \cup c_1)$ the set $\mathcal{S} \setminus B_{10\varepsilon}(B_1) \setminus B_{10\varepsilon}(C_1)$ is a regular tripod.
        \item [(ii)]  Every $\mathcal{S} \in \mathbb{M}([Y_{up}Y_{down}] \cup b_2 \cup c_2)$ is a regular tripod outside $B_{10\varepsilon}(B_2) \cup B_{10\varepsilon}(C_2)$.
    \end{itemize}
\end{lemma}

\begin{lemma}
    \label{firstlemma}
   There exists $\mathcal{S} \in \mathbb{M}(\{Y_1\} \cup b_1 \cup c_1)$ which is symmetric with respect to $l_1$.
\end{lemma}

\begin{proof}
Let $F$ be a point at the ray $[Y_1T_1)$ such that $|Y_1F| = 1 + \frac{3}{2}\lambda$ and denote by $DEF$ the equilateral triangle such that
$Y_1$ is the middle of the segment $[DE]$ and $B_1C_1$ is parallel to $DE$.
Consider segments $[Z_lZ_r] \subset [DF]$ and $[V_lV_r] \subset [EF]$ such that 
$|Z_lZ_r| = [V_lV_r] = \lambda$ and $Z := DF \cap T_1B_1, V := EF \cap T_1C_1$ are centers of the segments.
Note that $l_1$ is a symmetry axis of $DEF$ and $[Z_lZ_r]$, $[V_lV_r]$ are also symmetric with respect to $l_1$.

By Lemma~\ref{lemma0}(i) every minimal set $\mathcal{S}$ is a regular tripod $Y_1B_1'C_1'$ out of $B_{10\varepsilon}(B_1) \cup B_{10\varepsilon}(C_1)$.
We claim that the tripod $Y_1B_1'C_1'$ intersects segments $[Z_lZ_r]$ and $[V_lV_r]$. 
Indeed, consider Cartesian coordinates in which $Y_1 = (0,0)$, $B_1 = (1 + \lambda/2, \sqrt{3}\lambda/2)$ and $C_1 = (1 + \lambda/2, -\sqrt{3}\lambda/2)$. Then $Z = (1 + 3\lambda/8, 3\sqrt{3}\lambda/8)$, 
$Z_l = (1 + 3\lambda/8 - \sqrt{3}\lambda/4, 3\sqrt{3}\lambda/8 + \lambda/4)$ and $Z_r = (1 + 3\lambda/8 + \sqrt{3}\lambda/4, 3\sqrt{3}\lambda/8 - \lambda/4)$. Since the center $T_1'$ of $Y_1B_1'C_1'$ lies inside triangle $Y_1B_1'C_1'$ it has 
$x$-coordinate smaller than the $x$-coordinate of $B_1'$ and $y$-coordinate smaller than the $y$-coordinate of $B_1'$.

\begin{figure}[h]
    \centering
    \includegraphics[width=0.7\textwidth]{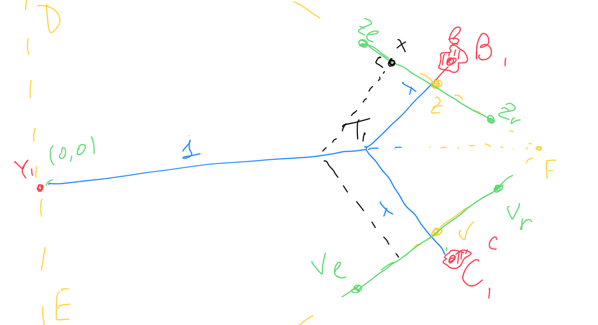}
    \caption{Picture to the proof of Lemma 1}
    \label{fig:lemma1}
\end{figure}

We consider the following auxiliary data for Steiner problem: $A_{mid} = [Z_lZ_r] \cup [V_lV_r] \cup \{Y_1\}$, 
$A_{up} = [Z_lZ_r] \cup b$, $A_{down} = [V_lV_r] \cup c$.
By the results from~\cite{paolini2013existence}, mentioned in the introduction every $\mathbb{M}(A_i)$ is not empty.
Segments $[Z_lZ_r]$ and $[V_lV_r]$ split every $\mathcal{S} \in \mathbb{M}(\{Y_1\} \cup b \cup c)$ on 3 parts which connect $A_{mid}$, $A_{up}$ and $A_{down}$, so
\begin{equation}
\H(\mathcal{S}) \geq \H(\mathcal{S}_{mid}) + \H(\mathcal{S}_{up}) + \H(\mathcal{S}_{down}),
\label{infactequal}
\end{equation}
where $\mathcal{S}_i \in \mathbb{M}(A_i)$. We claim that in fact the equality in~\eqref{infactequal} holds.

It is known (see a barycentric coordinate system) that the sum of distances from a point from a closed equilateral triangle to the sides does not depend on a point.
Thus $\mathbb{M}(A_{mid})$ is a set of regular tripods, and every such tripod is symmetric with respect to $l_1$. Moreover for every point $x \in [Z_lZ_r]$ there is a unique regular tripod $\mathcal{S}_x \in \mathbb{M}(A_{mid})$ and $\mathcal{S}_x$ is orthogonal to $[Z_lZ_r]$ at $x$.

Now consider any $\mathcal{S}_{down} \in \mathbb{M}(A_{down})$. Let $\mathcal{S}_{up}$ be a set, symmetric to $\mathcal{S}_{down}$ with respect to $l_1$; clearly $\mathcal{S}_{up} \in \mathbb{M}(A_{up})$.
For $x \in \mathcal{S}_{down} \cap [V_lV_r]$ the set $\mathcal{S}_x \cup \mathcal{S}_{up} \cup \mathcal{S}_{down}$ connects $\{Y_1\} \cup b \cup c$, and reaches the equality in~\eqref{infactequal}, so
$\mathcal{S}_x \cup \mathcal{S}_{up} \cup \mathcal{S}_{down}$ is a Steiner tree for $\{Y_1\} \cup b \cup c$. By the construction it is symmetric with respect to $l_1$.

\end{proof}

\begin{lemma}
    \label{vert}
    There exists $\mathcal{S} \in \mathbb{M}([Y_{up}Y_{down}] \cup b_2 \cup c_2)$ which is symmetric with respect to $l_2$.
\end{lemma}
\begin{proof}
    By Lemma~\ref{lemma0}[(ii)] every Steiner tree $\mathcal{S}$ coincides with a regular tripod outside $B_{10\varepsilon}(B_2) \cup B_{10\varepsilon}(C_2)$. Clearly its longest segment is perpendicular to $Y_{up}Y_{down}$. 
    We want to show that it touches $Y_{up}Y_{down}$ in $Y_2$, id est one of three segments is a subset of $l_2$.
    Assume the contrary, then $l_2 \cap \mathcal{S}$ is a point, denote it by $L$ and let $n \parallel B_2C_2$ be the line containing $L$. Then $n$ divides $\mathcal{S}$ into three connected components, denote them by $\mathcal{S}_Y$, $\mathcal{S}_b$ and $\mathcal{S}_c$, respectively.
    
\begin{figure}[h]
    \centering
    \includegraphics[width=0.5\textwidth]{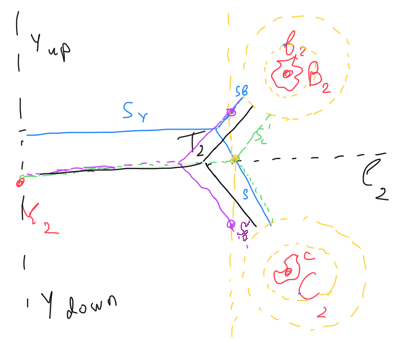}
    \caption{Picture to the proof of Lemma 2}
    \label{fig:lemma2}
\end{figure}

    Without loss of generality $L$ belongs to $\mathcal{S}_c$.
    Let us construct another competitors $\mathcal{S}_1$ and $\mathcal{S}_2$ connecting $[Y_{up}Y_{down}]$, $b$ and $c$. Let $\mathcal{S}_1 = [Y_2L] \cup \mathcal{S}_c \cup \mathcal{S}_c'$ where $\mathcal{S}_c'$ is a reflection of $\mathcal{S}_c$ with respect to $l_2$.
    Put $h:= \dist (Y_2, \mathcal{S}_Y \cap [Y_{up}Y_{down}])$. Thus 
    \[
    \H(\mathcal{S}_1) = \H(\mathcal{S}_Y) -\sqrt 3 h+ 2\H(\mathcal{S}_c). 
    \]
    Let $\mathcal{S}_2:= \mathcal{T} \cup \mathcal{S}_b \cup \mathcal{S}_b'$ where $\mathcal{S}_b'$ is a reflection of $\mathcal{S}_b$ with respect to $l_2$ and $\mathcal{T}$ is a regular tripod connecting $Y_2$ with $n \cap \mathcal{S}_b$ and $n \cap \mathcal{S}_{b'}$. Thus
    \[
    \H(\mathcal{S}_2)= \H(\mathcal{S}_Y) + \sqrt 3 h +2\H(\mathcal{S}_b).
    \]
    Since $\mathcal{S}$ is a Steiner tree, one has $\H(\mathcal{S}) \leq \H(\mathcal{S}_1)$, $\H(\mathcal{S}) \leq \H(\mathcal{S}_2)$ and clearly $\H(\mathcal{S})=\frac{\H(\mathcal{S}_1) + \H(\mathcal{S}_2)}{2}$. Then $\H(\mathcal{S})=\H(\mathcal{S}_1)=\H(\mathcal{S}_2)$, and so $\mathcal{S}_1, \mathcal{S}_2$ belong to $\mathbb{M}([Y_{up}Y_{down}] \cup b \cup c)$.
    As $\mathcal{S}_1$ and $\mathcal{S}_2$ are symmetric with respect to $l_2$ the statement is proved.

\end{proof}

Now we are ready to prove Theorem~\ref{newMainT}, i.e. to show that for $\lambda_j = \lambda < \frac{1}{300}$ 
the set $\Sigma(\Lambda)$ is a Steiner tree for the set of terminals $A_\infty$.

Let $b_1$ and $c_1$ be the subsets of terminals which are descendants of $y_2$ and $y_3$, respectively. 
Since $\varepsilon = \lambda^2 + \lambda^3 + \dots + \lambda^k + \dots$ we have $b_i \subset B_\varepsilon(B_i)$, $c_i \subset B_\varepsilon(C_i)$.
Applying Lemma~\ref{firstlemma} for $Y_1 = y_0$, $B_1 = y_2$, $C_1 = y_3$, $b_1$ and $c_1$ we show that there is a Steiner tree for $A_\infty$ which is symmetric with respect to the line $(y_0y_1)$.

Let $[Z_lZ_r]$ and $[V_lV_r]$ be the segments from the previous application of Lemma~\ref{firstlemma}.
Now define $b_2$ and $c_2$ as descendants of $y_4$ and $y_5$, respectively.
Then applying Lemma~\ref{vert} to $[Y_{up}Y_{down}] = [Z_lZ_r]$, $B_2 = y_4$, $C_2 = y_5$, $b_2$ and $c_2$ (this data is similar to the required with the scale factor $\lambda$) we show that there is a Steiner tree containing $[y_0y_1]$ and branching at $y_1$ (because $y_1$ belongs to the axis of symmetry of $b$ and $c$).

Now since $\lambda_i$ is constant, the upper and lower components of $\Sigma(\Lambda) \setminus [y_0y_1]$ are similar (with the scale factor $\lambda$) to $\Sigma(\Lambda)$.
So the second application of Lemmas~\ref{firstlemma} and~\ref{vert} shows that there is a Steiner tree containing $[y_0y_1] \cup [y_1y_2] \cup [y_1y_3]$.
This procedure recovers $\Sigma(\Lambda)$ step by step; so after $k$-th step we know that the length of every Steiner tree for $A_\infty$ is at least 
\[
\sum_{i=0}^{k-1} (2\lambda)^i.
\]
Thus the length of every Steiner tree for $A_\infty$ is at least the length of $\Sigma(\Lambda)$ which implies $\Sigma(\Lambda) \in \mathbb{M}(A_\infty)$.

\paragraph{Acknowledgements.} The authors are grateful to Fedor Petrov and Pavel Prozorov for checking early versions of the proof.

\bibliographystyle{plain}
\bibliography{main}

\section{Appendix}

\begin{proof}[Proof of Lemma~\ref{lemma0}] In this proof $i \in \{1,2\}$. Suppose that $\mathcal{S}$ intersects every circle $\partial B_\rho(B_i)$ in at least 2 points for $\varepsilon \leq \rho \leq 10\varepsilon$.
Then we may replace $\mathcal{S}$ with a shorter competitor as follows. 
Put $\mathcal{S}_b = \mathcal{S} \cap B_{\varepsilon}(B_i)$.
By the definition and the co-area inequality
\[
\H(\mathcal{S}) \geq \H(\mathcal{S}_b) + 2\cdot 9\varepsilon + \H(\mathcal{S}_i),
\]
where $\mathcal{S}_1 \in \mathbb{M} (\{Y_1\} \cup \partial B_{10\varepsilon}(B_1) \cup c_1)$,
$\mathcal{S}_2 \in \mathbb{M} ([Y_{up}Y_{down}] \cup \partial B_{10\varepsilon}(B_2) \cup c_2)$.
Now take $\mathcal{S}_i \cup \mathcal{S}_b \cup \partial B_{\varepsilon}(B_i) \cup \mathcal{R}_B$, where $\mathcal{R}_B$ is the radius connecting $\mathcal{S}_i$ with $\partial B_{\varepsilon}(B_i)$. 
The length of this competitor is 
\[
\H(\mathcal{S}_b) +  2 \pi \varepsilon  +  9\varepsilon + \H(\mathcal{S}_i),
\]
which gives a contradiction since $2\pi < 9$. The symmetric construction shows that the situation that $\mathcal{S}$ intersects every circle $\partial B_\rho(C_i)$ in at least 2 points for $\varepsilon \leq \rho \leq 10\varepsilon$ is also impossible.

\begin{figure}[h]
    \centering
    \includegraphics[width=\textwidth]{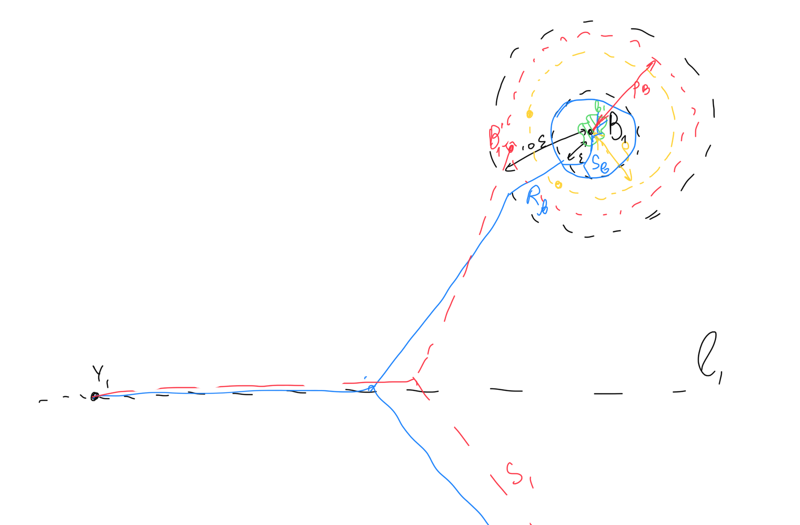}
    \caption{Picture to the proof of Lemma 0}
    \label{fig:Lemma0}
\end{figure}

Thus there are $\rho_b$, $\rho_c \in [\varepsilon,10\varepsilon]$ such that $\mathcal{S} \cap \partial B_{\rho_b}(B_i)$ is a point
$B_i'$ and $\mathcal{S} \cap \partial B_{\rho_c}(C_i)$ is a point $C_i'$.
Clearly, $\mathcal{S} = \mathcal{S}_i \cup \mathcal{S}_b \cup \mathcal{S}_c$, where 
$\mathcal{S}_b = \mathcal{S} \cap B_{\rho_b}(B_i)$, $\mathcal{S}_c = \mathcal{S} \cap B_{\rho_c}(C_i)$ and
$\mathcal{S}_1 \in \mathbb{M} (\{Y_1\} \cup \{B_i'\} \cup \{C_i'\})$,
$\mathcal{S}_2 \in \mathbb{M} ([Y_{up}Y_{down}] \cup \{B_i'\} \cup \{C_i'\})$.
Clearly $\mathcal{S}_i$ is a tripod or the union of 2 segments.
We claim that $\mathcal{S}_i$ is a tripod. By the triangle inequality
\begin{equation}
| \H ([T_iB_i] \cup [T_iC_i]) - \H ([T_iB_i'] \cup [T_iC_i']) |  < 20 \varepsilon.
\label{diff}    
\end{equation}

Now let us prove item (i). By~\eqref{diff} the length of the (non-regular) tripod $[T_1Y_1] \cup [T_1C_1'] \cup [T_1B_1']$ connecting $Y_1,B_1'$ and $C_1'$ is at most 
$1 + 2\lambda + 20\varepsilon$.
By the same reason the length of 2-segments is at least
\[
\sqrt{1 + \lambda + \lambda^2} + \sqrt{3} \lambda - 30\varepsilon > 1 + \left( \frac{1}{2} + \sqrt{3} \right) \lambda - 30 \varepsilon.
\]
Recall that $\varepsilon = \frac{\lambda^2}{1-\lambda}$; it is straightforward to check that
\[
1 + \left( \frac{1}{2} + \sqrt{3} \right) \lambda - 30 \varepsilon > 1 + 2\lambda + 20\varepsilon
\]
for $\lambda < 1/300$. So we show that $\mathcal{S}_1$ contains a tripod connecting $Y_1,B_1'$ and $C_1'$; by the minimality argument it is regular.

Let us deal with item (ii). By~\eqref{diff} the length of the (non-regular) tripod $[T_2Y_2] \cup [T_2C_2'] \cup [T_2B_2']$ connecting $Y_2, B_2'$ and $C_2'$ is at most 
$1/4 + 2\lambda + 20\varepsilon$. Again 2-segment construction has the length at least
\[
1/4 + \lambda/2 + \sqrt{3}\lambda - 30\varepsilon.
\]
The rest calculations coincide with the first item.

\end{proof}

\end{document}